\documentclass[onefignum,onetabnum]{siamart190516}

\usepackage{lipsum}
\usepackage{amsfonts}
\usepackage{graphicx}
\usepackage{epstopdf}
\usepackage{algorithmic}
\ifpdf
  \DeclareGraphicsExtensions{.eps,.pdf,.png,.jpg}
\else
  \DeclareGraphicsExtensions{.eps}
\fi

\newsiamremark{remark}{Remark}
\newsiamremark{hypothesis}{Hypothesis}
\crefname{hypothesis}{Hypothesis}{Hypotheses}
\newsiamthm{claim}{Claim}

\headers{Maximizing The Distance To a Far Enough  Point}{M. Costandin, B. Gavrea, B. Costandin}

\title{Maximizing the distance to a ``far enough'' point over the intersection of hyper-disks 
}

\author{Costandin Marius\thanks{General Digits 
  (\email{costandinmarius@gmail.com}, \url{http://www.generaldigits.com}).}
\and Gavrea Bogdan \thanks{Technical University of Cluj-Napoca 
  (\email{bogdan.gavrea@math.utcluj.ro}).}
\and Costandin Beniamin \thanks{Technical University of Cluj-Napoca 
  (\email{bcostandin@yahoo.com}).}
}

\usepackage{amsopn}

\ifpdf
\hypersetup{
  pdftitle={Maximizing the distance to a ``far enough'' point over the intersection of hyper-disks},
  pdfauthor={Costandin Marius, Gavrea Bogdan }
}
\fi

\usepackage{algorithmic}

\begin{document}

\maketitle

\begin{abstract}
  We present a novel feasibility criteria for the finite intersection of convex sets given by inequalities. This criteria allows us to easily assert the feasibility by analyzing the unconstrained minimum of a specific convex function, that we form with the given sets. 
 Next an algorithm is presented which extends the idea to a particular non-convex case: assert the inclusion of the finite intersection of a set of hyper-disks with equal radii in another hyper-disk with a different radius. 
\end{abstract}



\section{Introduction}


In this paper we present a novel framework for asserting the feasibility of the intersection of convex sets. Our approach is to synthesize the information in the given convex sets in a non-smooth convex function whose unconstrained minimizer can be used to assert the feasibility of the intersection. Algorithms in literature for such  a problem, namely convex feasibility,  exist,  see \cite{ref1}, \cite{ref2}, \cite{ref3} or  \cite{ref4} for example. Our contribution here is the presentation of a simple and elegant criteria. Unlike the references above, we do not focus on the convex minimization problem itself, but on the formation of the convex function to be minimized and   on the interpretation of the resulting minimizer.

 Next we extend the presented method to a particular case of mathematical programming: the assertion of the inclusion 
 of an intersection of hyper-disks in another hyper-disk. We are able to give meaningful results under some 
 requirements regarding the distance between the center of the outsidehyper-diskl and the the intersection of the hyper-
 disks with equal radii. 

We will use throughout the paper the symbol $d(\cdot, \times)$ where $\cdot$ can be a point and $\times$ can be a point or a convex set of points, to designate the Euclidean distance between $\cdot$ and $\times$.  For a vector $u\in\mathbb{R}^n$, $u =\left(u_1, ...,u_n\right)^T$  and $r>0$, we denote by $B(u,r)$ the open ball centered at $u$ and of radius $r$. We also denote by $\|u\|$, $\|u\|^2 = u^T u$,  the Euclidean norm of the vector $u$.

\subsection{Convex domains of interest}
Let $x \in \mathbb{R}^{n}$, $n,m \in \mathbb{N}_+$ and let $g_k : \mathbb{R}^{n} \to \mathbb{R}$ be convex functions for $k \in \{1, \hdots, m\}$. We define the convex sets:
\[
S_k = \left \{ x \in \mathbb{R}^{n} \biggr | g_k(x) \leq 0 \right \} 
\] 
and we are interested if the set 
\begin{equation}\label{E1.3}
\mathcal{S} = \bigcap_{k=1}^m S_k
\end{equation} i
s empty or not. For this we define the following function:
\[
\widetilde{G}(x) = \sum_{k=1}^m g_k^+(x)
\] where 
\begin{equation}\label{eq:gkplus}
g_k^+(x) = \begin{cases}
g_k(x), \hspace{1cm} & g_k(x) \geq 0\\
0, & g_k(x) \leq 0 
\end{cases}
\end{equation}

\section{Main results}
In this section we present a novel feasibility criteria for the finite intersection of certain convex sets. We give a test for the inclusion of an intersection of hyper-disks  into another hyper-disk. The method of alternating projections, \cite{BauschkeBorwein1996}, is the typical choice  in the context of finding a feasible solution at the intersection of convex sets. In this paper, we give a projection-free method for solving set intersection problems. Our approach reformulates the feasibility problem as a non-smooth convex minimization problem.

\subsection{Convex feasibility}\label{sec:cvxfs}


The following result is a characterization of the set $\mathcal{S}$  in terms of a global minimum  of $\widetilde{G}(x)$.
\begin{lemma}Let 
\begin{equation}\label{eq:xstar1}
x^{\star} = \underset{x \in \mathbb{R}^n}{\mathrm{argmin}} \; \widetilde{G}(x).
\end{equation}
 Then the following are equivalent:
\begin{enumerate}
\item The set $\mathcal{S}$ is not empty, i.e $\exists\; x^0 \in \mathbb{R}^n$ such that 
\[
 g_k(x^0) \leq 0 \hspace{1cm}  \forall k  \in \{1, \hdots, m\}\nonumber 
\]
\item The point $x^\star$ defined by \cref{eq:xstar1} satisfies
\[
g_k(x^{\star}) \leq 0 \hspace{1cm} \forall k  \in \{1, \hdots, m\} \nonumber 
\]
\end{enumerate}
\end{lemma}
\begin{proof} 
The part  $2 \Rightarrow 1$ follows immediately from $g_k(x^{\star}) \leq 0$ for all $k \in \{1, \hdots, m\}$ which implies $x^{\star} \in \mathcal{S}$ and therefore  $\mathcal{S}\neq \emptyset$. 
To prove $1 \Rightarrow 2$, let $x^0$ such that $g_k(x^0) \leq 0$ for all $k \in \{1, \hdots, m\}$ and assume that $\exists k$ such that $g_k(x^{\star}) > 0$. This implies 
\[
0 = \widetilde{G}(x^0) < \widetilde{G}(x^{\star}) 
\]  
which contradicts   the fact that $x^{\star}$ is a global minimum of $\widetilde{G}$. 
\end{proof}

\begin{remark}
The simple result above shows that the feasibility of the intersection of $m$ convex sets (sub-level sets of convex functions) can be asserted by examining the global minimum of a non-smooth convex function.
\end{remark}

\subsection{Test for the inclusion of an intersection of spheres into  another sphere}
We want to solve the following non-convex optimization problem:
\begin{eqnarray} \label{E2.2}
\max & \|x-c\|^2  & \nonumber \\
\mathrm{s.t}\;\; & \|x-c^k\|^2 \leq R^2,\;\;  & \forall k\in\{1, \hdots, m\}, \label{E2.3}
\end{eqnarray} 
where $c^k, c\in \mathbb{R}^n$ and $R\in \mathbb{R}$, $R>0$. Problem (\ref{E2.3}) is equivalent to finding a point in the intersection of the spheres centered at $c^k$ and of radius $R$ which is the furthest away from the point $c$. 
The \emph{S-procedure}, \cite{BoydGhaouiFeron94},
is a well known algorithm for programs with quadratic objective and quadratic constraints. Our approach is different and focuses on solving a non-smooth minimization problem.  

We consider  the following sets:
\begin{align} \label{E2.7}
\mathcal{B}_0 &= \overline{B}(c,r) =  \left\{ x \in \mathbb{R}^n \biggr| \| x-c\|^2  \leq r^2 \right\},  \nonumber \\
\mathcal{B}_k &= \overline{B}(c^k,R) = \left\{ x \in \mathbb{R}^n \biggr| \|x-c^k\| \leq R^2 \right\}, \nonumber \\
\mathcal{C}_1 &= \bigcap_{k=1}^m \mathcal{B}_k,  \hspace{1cm} \mathcal{C}_0 = \mathcal{B}_0
\end{align}
 for some $R,r > 0$. In order to solve the problem (\ref{E2.3}), we keep $R$ fixed and  design  a test which can assert if $\mathcal{C}_1 \subseteq \mathcal{C}_0$ for various values of $r$.

We start by defining the functions:
\begin{align}\label{E2.15a}
f_k(x) &=  \|x-c^k\|^2 - R^2  \nonumber \\
f(x) &=  \|x-c\|^2- r^2
\end{align} 
and the function $G_k(x)$, given by 
\[
G_k(x) = f_k(x) - f^{-}(x) + \sum_{i = 1, i\neq k}^m f_i^{+}(x)
\]
for $k\in \{1,...,m\}$. Here $f_i^+(x)$ are defined according to (\ref{eq:gkplus}), while 
$f^-(x)$ doesn't stand for the classical definition of the negative part of $f(x)$,  but rather 
$f^-(x) :=\min\{f(x),0\}\le 0$. 
\begin{remark}
It can be seen that $G_k$ is a convex function. First the \emph{``sum''}--term $\sum_{i = 1, i\neq k}^m f_i^{+}(x)$ is convex, since each term in the sum is convex. On the other hand, the remaining term of $G_k(x)$, namely 
$f_k(x)- f^{-}(x)$, can be written as
\[
f_k(x)- f^{-}(x)= f_k(x) - f(x) + f(x) - f^{-}(x) = f_k(x) - f(x) + f^{+}(x)
\]
which is convex since it is  the sum of the convex function $f^+(x) $ and the affine function $f_k(x) - f(x)$. 
\end{remark} 

We take $G(x)$ to be the maximum of $G_k(x)$, when $k$ ranges from $1$ to $m$. That is, 
\[
G(x) = \max \hspace{0.1cm} \left\{ G_k(x) \biggr | k \in \{ 1, \hdots, m\} \right\} =\underset{k=\overline{1,m}}{ \max} G_k(x)
\]
\begin{remark}
We note that, since $G:\mathbb{R}^n \to \mathbb{R}$ is defined  as the pointwise maximum of the convex functions $G_k:\mathbb{R}^n\to \mathbb{R}$, it follows that $G$ is convex. 
\end{remark}

Finally we define $x^\star$ tobe the global minimizer of $G(x)$, i.e., 
\begin{equation}\label{eq:xstar}
x^{\star} = \mathop{\text{argmin}}_{x \in \mathbb{R}^n} \hspace{0.2cm} G(x) 
\end{equation} 
Before giving our main result,we present a few simple but usefull lemmas.
\begin{lemma}\label{L2.5}
Let $a, b \in \mathbb{R}^n$ and $r > 0$ such that $b \not\in B(a,r)$. Then $\forall x \in B(a,r)$ the following
\[
(x - b)^T  (a - b) > 0
\]
holds.
\end{lemma}
\begin{proof}
Using the Euclidean norm properties over $\mathbb{R}^n$, we write
\begin{eqnarray}
\|x-a\|^2 & = & \|(x-b)+(b-a)\|^2\nonumber\\
	    &  = & \|x-b\|^2+\|b-a\|^2- 2(x-b)^T(a-b).\label{eq:auxL2.5}
\end{eqnarray}
For $x\in B(a,r)$, $b\notin B(a,r)$, we have $\|x-a\|^2<r^2$ and $\|b-a\|^2\ge r^2$. Combining these  together with $\|x-b\|^2\ge 0$ in (\ref{eq:auxL2.5}), leads to $(x-b)^T(a-b) <0$ and concludes the proof. 
\end{proof}

\begin{lemma} \label{L2.6}
 Let $x \in \mathcal{C}_1$, with $\mathcal{C}_1$ defined by (\ref{E2.7}). Then for $y \in \mathbb{R}^n$ such that  $d(y,\mathcal{C}_1) > R$ one has 
\[
(x - y)^T (c^k - y) > 0, \;\; \forall k \in \{1, \hdots, m\} 
\]
\end{lemma}
\begin{proof}
For $x \in \mathcal{C}_1$, one has $d(x,c^k) \leq R$  and therefore  $c^k \in B(x,R)$. From $d(y,\mathcal{C}_1) > R$, it  follows that $d(x,y) > R$, hence $y \not\in B(x,R)$. Applying \cref{L2.5}, with $a:= x$, $b:=y$, $r:=R$,  and $x:= c^k$,  one obtains the desired conclusion. 
\end{proof}

\begin{lemma} \label{L2.7}
Let $z,y,c^1,c^2 \in \mathbb{R}^n$ with $\| y - c^1\| = \| y - c^2\| $. Assume, without loss of generality, that $\|z - c^1\|^2 \geq \|z - c^2\|^2 $ then
\[
\|y + t  (z-y) - c^1\|^2 \geq \| y + t  (z-y) - c^2 \|^2, \hspace{1cm} \forall t \geq 0. 
\]
\end{lemma}
\begin{proof}
Let 
\[
h(t) = \|y + t  (z-y) - c^1\|^2 - \|y + t  (z-y) - c^2\|^2.
\] 
From the identity above, it  can be  seen that $h(t)$ is a polynomial of degree at most $1$ in $t$. Since
$\|y -c^1\|=\|y -c^2\|$ gives $h(0) = 0$ and $\|z -c^1\|\ge \|z -c^2\| $ gives $h(1) \ge h(0) = 0$, it follows that $h(t)$ is a non-decreasing first order polynomial in $t$ and therefore
$$
 h(t) \ge 0=h(0), \;\; \forall t\ge 0,
$$
which completes the proof. 
\end{proof}

\begin{lemma} \label{L2.8}
Let $y, c^1, \hdots, c^m \in \mathbb{R}^n$ and $v \in \mathbb{R}^n$ such that  $\|v\| = 1$. Let $p\in \{1,...,m-1\}$ be such that 
\begin{equation}\label{eq:2.17}
\| y - c^{i}\| = \| y - c^{j}\| > \| y - c^{l}\| \hspace{1cm}
\end{equation} 
for all $i,j \in \{1, \hdots, p\}$ and $l \in \{p+1, \hdots, m\}$. Then $\exists k_v \in \{1,...,p\}$ and $\delta_v > 0$ such that for all $i \in \{1, \hdots, m\} $ one has
\begin{equation}\label{eq:2.18}
\|y + t  v- c^{k_v}\| \geq \| y + t  v - c^i \|\hspace{1cm} \forall t \in (0,\delta_v),
\end{equation} 
which is stating that  there is a small segment  starting at $y$ in the direction of $v$, such that for all the  points on this segment, $c^{k_v}$ remains the furthest away. For  the case $p=m$,  (\ref{eq:2.18}) holds without  any additional requirements. 
\end{lemma}
\begin{proof}
First, we consider  the case $p\in \{1,...,m-1\}$. We define  $\rho := \|y - c^{1}\|=\hdots = \|y - c^{p}\|$. Let  $\delta > 0$ and  $ z \in B(y,\delta)$. The triangle inequality gives 
\begin{eqnarray*}
\|z-c^k \| & \ge & \|c^k-y\| -\|z-y\|,\\
\|z - c^i \| & \le  & \|c^i-y\| + \| z- y \|.
\end{eqnarray*}
Using the above inequalities with arbitrary $k\in \{1,...,p\}$ and $i \in \{p+1,...,m\}$,  gives 
\begin{equation}\label{eq:2.19}
\begin{cases}
d(z,c^k) \geq \rho - \delta, \\
d(z, c^i) \leq \eta + \delta.
\end{cases}
\end{equation} 
Following (\ref{eq:2.19}), we will pick $\delta>0$ such that $\rho-\delta > \eta+\delta$. Since (\ref{eq:2.17}) implies 
$\rho-\eta>0$, it follows that any  $\delta \in \left( 0, \frac{\rho-\eta}{2}\right)$ will satisfy this requirement. Thus, for any 
$\delta \in \left( 0, \frac{\rho-\eta}{2}\right)$ and any $z\in B(y,\delta)$, we have
\begin{equation}\label{eq:2.20}
d(z,c^{k}) > d(z,c^i) \hspace{0.5cm} \forall k \in \{1, \hdots, p\} ,\hspace{0.1cm} \forall i \in \{p+1, \hdots, m\}. 
\end{equation}  
 Let $\delta_v = \frac{\delta}{2}$,  $z = y + \delta_v  v$  and  $k_v = \underset{k \in \{1, \hdots, p\}} {\mathop{\text{argmax}}} \| z - c^k\| $.  For the points $c^k$, $k\in \{1,...,p\}$, we apply   \cref{L2.7} to obtain 
  \begin{equation}\label{eq:aux1}
   \| y +(t \delta_v)  v- c^{k_v} \|^2 \ge  \| y +(t \delta_v)  v- c^{k} \|^2, \; \forall t\ge 0, \; \forall k \in \{1,...,p\}.
  \end{equation}
  On the other hand, for the points $c^i$, $i\in \{p+1,...,m\}$ we let  $z := y+t v$  in (\ref{eq:2.20}) which gives 
  \begin{equation}\label{eq:aux2}
  	 \| y +t  v- c^{k_v} \|^2 >  \| y +t v- c^{i} \|^2\;\; \forall i\in \{p+1,...,m\}, \; \forall t\in (0,\delta_v).
  \end{equation}
  Combining (\ref{eq:aux1}) and (\ref{eq:aux2}) leads to the desired conclusion (\ref{eq:2.18}). For the case $p=m$, 
  (\ref{eq:2.18}) follows immediatelly. 
\end{proof}

The following theorem represents our {\bf main result}. This is a localization result for $x^\star$ using the sphere intersection $\mathcal{C}_1$ and the \emph{``outside''} sphere $\mathcal{C}_0$. 
\begin{theorem}\label{T2.9}
If $d(C, \mathcal{C}_1) > R$ then 
\begin{equation}\label{eq:2.23}
\mathcal{C}_1 \setminus \mathrm{int}(\mathcal{C}_0) \neq \emptyset \iff x^{\star} \in \mathcal{C}_1 \setminus \mathrm{int}(\mathcal{C}_0)
\end{equation}
\end{theorem} 
\begin{proof} Clearly the implication $x^{\star} \in \mathcal{C}_1 \setminus \mathrm{int}(\mathcal{C}_0) \Rightarrow
\mathcal{C}_1 \setminus \mathrm{int}(\mathcal{C}_0) \neq \emptyset $ is trivial. 
 We now  assume that $\mathcal{C}_1 \setminus \mathrm{int}(\mathcal{C}_0)\neq \emptyset$ and let $x\in \mathcal{C}_1$. It follows that $\|x -c^k \| >R^2$ for some $k \in \{1,...,m\}$ or equivalently $f_k(x) >0$ for some  $k \in \{1,...,m\}$.
 From the definitions of $f^-$ and $f_i^+$, we have $-f^-(x) \ge 0$ and $f_i^+(x) \ge 0$. Combining this with $f_k(x) >0$, leads to the fact that for $x\notin \mathcal{C}_1$ we have $G_k(x) >0$.  On the other hand if $x \in \mathcal{C}_1\setminus \mathrm{int}(\mathcal{C}_0)$, we have  $- f^-(x) = 0$, $f_k(x) \le 0$, $\forall k \in \{1,...,m\}$, implying $G(x) \le 0$. 
 
 From the observations above, it follows that $x^\star\in \mathcal{C}_1$. Next,  we will show that 
 $x^\star \notin \mathrm{int}(\mathcal{C}_1 \cap \mathcal{C}_0)$, leading to the desired conclusion. Let $y\in \mathrm{int}(\mathcal{C}_1 \cap \mathcal{C}_0)$. It follows that there exists $\delta_y>0$ such that $B(y, \delta_y) \subseteq \mathrm{int}(\mathcal{C}_1 \cap \mathcal{C}_0)$. We can assume without loss of generality that $\exists p\in \{1,...,m-1\}$ such that 
 $$
 \| y - c^1\| = ...= \| y - c^p\| > \| y - c^l \|, \; \forall l \in \{p+1,...,m\}. 
 $$
 This implies 
 $$
 	G(y) = G_1(y) =...= G_p(y). 
 $$
 From \cref{L2.8}, it follows that $\forall \; v\in \mathbb{R}^n$ with $\|v \| = 1$, $\exists k_v \in \{1,..., p\}$ and 
 $\delta_v>0 $ such that 
 \begin{equation}\label{eq:2.27}
 G(y + t v ) = G_{k_v}(y + t  v) \hspace{1cm} \forall t \in [0,\delta_v).
 \end{equation}
 Let $\delta:= \min\{\delta_y, \delta_v\}$, $v= \frac{y-c}{\|y-c\|}$ and $z=y+\frac{\delta}{2}$.   Clearly $z \in  \mathrm{int}(\mathcal{C}_1 \cap \mathcal{C}_0)$. Let $h(t): = G(y+tv)$, $\forall t \in [0, \delta_v)$. From (\ref{eq:2.27}), it follows that $h(t) = G_{k_v}(y+tv)$, or equivalently 
 \begin{equation}\label{eq:2.29}
 h(t) =  r^2 - \|y - c + t v\|^2 + \| y - c^{k_v} + t v\|^2 - R^2, \; \forall t \in [0, \delta_v).
 \end{equation}
 Differentiating \cref{eq:2.29} with respect to $t$ gives 
 \begin{eqnarray} 
h'(t) &= -(y - c + t v)^T \ v + (y - c^{k_v} + t  v)^T  v \nonumber \\
& = -(c^{k_v} -c)^T  \frac{y-c}{\|y - c\|}. \label{eq:2.30}
\end{eqnarray}
Since $d(c, \mathcal{C}_1)>R$, it follows from \cref{L2.6} and (\ref{eq:2.30}) that $h'(t)<0$, $\forall t\in [0, \delta_v)$ implying that $h(t)$ is strictly decreasing. Therefore $h(0) >h (\frac{\delta}{2})$, which is equivalent to 
$G(z) < G(y)$, $z = y +\frac{\delta}{2} v \in \mathrm{int}(\mathcal{C}_0 \cap \mathcal{C}_1)$. It follows that $x^\star =\underset{x\in \mathbb{R}^n}{\mathrm{argmin}\; G(x)} \notin \mathrm{int}(\mathcal{C}_1\cap \mathcal{C}_0)$. Since 
$\mathcal{C}_1$ can be partitioned as
$$
  \mathcal{C}_1 = \mathcal{C}_1\setminus \mathcal{C}_0 \cup \mathrm{int}(\mathcal{C}_1\cap \mathcal{C}_0)
  \cup \partial  (\mathcal{C}_1\cap \mathcal{C}_0) 
$$
and we showed that $x^\star\in \mathcal{C}_1$, $x^\star \notin \mathrm{int}(\mathcal{C}_1\cap \mathcal{C}_0)$, we 
have 
\begin{eqnarray*}
x^\star & \in & \mathcal{C}_1\setminus \mathcal{C}_0 \cup \partial  (\mathcal{C}_1\cap \mathcal{C}_0) \\
	  & \subseteq &    \mathcal{C}_1\setminus \mathcal{C}_0 \cup \partial \mathcal{C}_0,
\end{eqnarray*}
implying  that  $x^\star \in \mathcal{C}_1\setminus \mathrm{int}(\mathcal{C}_0)$. This concludes our proof. 

\end{proof}

\subsection{Short Complexity Analysis}
 \cref{T2.9} allows one to solve (\ref{E2.3}) if $d(c,\mathcal{C}_1) > R$. Indeed, let $x^0 \in \mathcal{C}_1$ (this can be found initially by the use of  Section \ref{sec:cvxfs} assuming that $\mathcal{C}_1 \neq\emptyset$) then one can show that $\mathcal{C}_1 \subseteq B(x^0, 2 R)$. Let $\underline{r} = R$ and $\bar{r} = 2 R + \|x^0 - c\|$. It is obvious that $\mathcal{C}_1 \setminus B(c, \underline{r}) \neq \emptyset$ and $\mathcal{C}_1 \setminus B(c, \bar{r}) = \emptyset$. 

We can now search for $r^{\star} \in [\underline{r}, \bar{r}]$ such that $\mathcal{C}_1 \setminus B(c, r^{\star} - \epsilon) \neq \emptyset $ and $\mathcal{C}_1 \setminus B(c, r^{\star}+\epsilon) = \emptyset$ for some arbitrarily fixed precision $\epsilon > 0$, using \cref{T2.9} and the bisection algorithm. 

From the computation complexity point of view, each bisection step involves the application of \cref{T2.9} for some $r \in [\underline{r}, \bar{r}]$. For this, one has to solve \cref{eq:xstar} to find $x^{\star}$. Once $x^{\star}$ is found, asserting its membership to $\mathcal{C}_1 \setminus B(c,r)$ involves computing $m+1$ distances in $\mathbb{R}^n$, that is $(m+1)  n$ flops (for the square of the distances) and comparing them to some real numbers, hence another $m+1$ flops. Finally the computational complexity analysis for each step is completed by analyzing the cost of finding $x^{\star}$. This basically involves an unconstrained minimization of a continuous, non-differentiable convex function. The starting point can be considered $x^0$ and the search radius can be taken $2 R$. There are various algorithms (of sub-gradient, \cite{stan_notes} or ellipsoid type, \cite{ref1}) which are known to have polynomial deterministic worst case complexity for such a problem. Let $\Lambda$ (a polynomial in $n,m, \log(R), -\log(\epsilon)$) denote the number of flops required to solve \cref{eq:xstar}. Then solving (\ref{E2.3})  requires
\[ 
\mathcal{O} \left( \left(\Lambda + (m +1 ) \cdot n \right) \cdot \log_2\left( \frac{R + \|X_0 - C\|}{\epsilon} \right) \right),
\] 
where $\epsilon > 0$ is the precision used to find $r^{\star}$. 

\section{An Application}

Let $\mathcal{S}\subseteq \mathbb{R}^n$ be a convex set and $c \in \mathbb{R}^n$. If for some $\delta > 0$, the intersection   $\mathcal{C}_1$  of $m \in \mathbb{N}$ hyper-disks of radius $R > 0$ is non-empty and   
\begin{enumerate}
\item[-] $\mathcal{C}_1 \subseteq \mathcal{S}$, 
\item[-] for all $x \in \mathcal{S}$ one has $B(x, \delta) \cap \mathcal{C}_1 \neq \emptyset$
\end{enumerate} 
with  $d(c, \mathcal{C}_1) > R$,
 then one can maximize the distance to $c$ over $\mathcal{S}$ up to $\delta$ precision due to the following theorem.

\begin{theorem}
Let 
\begin{eqnarray*}
V_s^2 &=& 
\max_{x \in \mathbb{R}^n} \hspace{0.1cm}  \| x - c\|^2 \hspace{0.5cm} \mathrm{s.t}  \hspace{0.3cm} x \in \mathcal{S}  \\
V_c^2 &= &\max_{x \in \mathbb{R}^n} \hspace{0.1cm}  \| x - c\|^2 \hspace{0.5cm} \mathrm{s.t}  \hspace{0.3cm} x \in \mathcal{C}_1.
\end{eqnarray*} 
Then 
\[
V_c \leq V_s \leq V_c + \delta.
\]
\end{theorem}
\begin{proof} 
Since $\mathcal{C}_1 \subseteq \mathcal{S}$ it is obvious that $V_c \leq V_s$. Next, let $x \in \mathcal{S}$ and $ y \in B(x,\delta) \cap \mathcal{C}_1$. Then
\[
\|x - c\| \leq \|x - y\| + \|y - c\| \leq \delta + \max_{y \in \mathcal{C}_1} \| y - c \| = \delta + V_c
\]
\end{proof}

\begin{remark} Please note that $V_c$ can be computed with the theory presented above. 
Finally, please note that it  is now easy to obtain $\hat{x} \in \partial \mathcal{S}$ such that 
\[
V_c \leq \| \hat{x} - c\| \leq V_c + \delta
\] 
by taking 
\[
\hat{x} = \mathop{\text{argmax}}_{x \in \mathcal{S}} \hspace{0.3cm} (x^{\star}_c - c)^T  (x - x^{\star}_c)
\] 
which is a linear program, with 
\[
x^{\star}_c = \mathop{\text{argmax}}_{x \in \mathcal{C}_1} \hspace{0.3cm} \| x - c\|^2
\]
\end{remark}

\section{Conclusion and future work}
As future work, one should study the numerical stability of the presented algorithm and  test  the algorithm 
 on some benchmark problems.   

\bibliographystyle{siamplain}

\begin{thebibliography}{xx}  

\bibitem{ref1} B. T. Polyak
\newblock Minimization Of Unsmooth Functionals
\newblock \emph{Moscow 1968}

\bibitem{ref2} B. T Polyak 
\newblock A general method for solving extremal problems.
\newblock DokE. Akad. Nauk SSSR. 174, 1, 33-36, 1967.

\bibitem{ref3} B.T. Polyak
\newblock Introduction to Optimization
\newblock Optimization Software New York

\bibitem{ref4} N. Parikh and S. Boyd
\newblock Proximal algorithms
\newblock Foundations and Trends in Optimization 1 123–231, 2013


\bibitem{stan_notes} S. Boyd
\newblock Subgradient Methods
\newblock Notes for EE364b, Stanford University, Spring 2013–14

\bibitem{BoydGhaouiFeron94}
S. Boyd, L. El Ghaoui, E. Feron and V. Balakrishnan
\newblock  Linear Matrix Inequalities in System and Control Theory
\newblock \emph{Society for Industrial and Applied Mathematics}, 1994

\bibitem{sahni_comp}
Sahni, Sartaj
\newblock Computationally Related Problems
\newblock \emph{SIAM J Comput}, vol. 3, nr. 4, 1974

\bibitem{ellipsoid_alg}
Robert G. Bland, Donald Goldfarb and
Michael J. Todd
\newblock The Ellipsoid Method: A Survey
\newblock \emph{Cornell University, Ithaca, New York}, 1981

\bibitem{BauschkeBorwein1996}
H. Bauschke, J. M. Borwein
\newblock{On Projection Algorithms for Solving Convex Feasibility Problems}
\newblock \emph{SIAM Review}, 38(3), 1996.

%
%
%
%
%

%

\end{thebibliography}

\end{document}